\documentclass[12pt,a4paper]{article}
\usepackage{latexsym,amssymb,amsfonts,amsmath,amsthm,nccmath,float,enumitem,setspace,bm,authblk}
\usepackage[usenames,dvipsnames]{xcolor}
\usepackage[hidelinks]{hyperref}
\usepackage[margin=1.65cm]{geometry}
\usepackage{bm}
\usepackage{booktabs}
\usepackage{verbatim}
\usepackage{diagbox}
\usepackage{makecell}
\usepackage{float}
\usepackage{ulem}
\allowdisplaybreaks[1]

\hypersetup{
	colorlinks,
	linkcolor={red!80!black},
	citecolor={blue!80!black},
	urlcolor={blue!80!black}
}

\newtheorem{thm}{Theorem}[section]

\newtheorem{pro}[thm]{Proposition}

\newtheorem{defi}[thm]{Definition}
\newtheorem{lem}[thm]{Lemma}
\newtheorem{core}[thm]{Corollary}

\newtheorem{example}{Example}[section]%

\def \leq {\leqslant}
\def \geq {\geqslant}

\def\div{\,\,\big|\,\,}

\def\Cay{{\sf Cay}}

\let\oldproofname=\proofname
\renewcommand{\proofname}{\rm\bf{\oldproofname}}

\begin{document}
	
	\title{On $2$-integral Cayley graphs}
	\author{Alireza Abdollahi $^{a}$\footnote{E-mail:
			a.abdollahi@math.ui.ac.ir }, Majid Arezoomand $^{b,c}$\footnote{Email: arezoomand@lar.ac.ir (corresponding author)}, Tao Feng $^{d}$\footnote{Email: tfeng@bjtu.edu.cn} and Shixin Wang $^{d,e}$\footnote{Email: sxwang@bjtu.edu.cn}\\
		{\small\em$^a$ Department of Pure Mathematics, Faculty of Mathematics and Statistics,}\\ 
		{\small\em University of Isfahan, Isfahan 81746-73441, Iran}\\
		{\small\em$^b$ University of Larestan, Larestan 74317-16137, Iran}\\
		{\small\em$^c$ Department of Mathematics, Faculty of Science, Shahid Rajaee Teacher Training
			University, Tehran, 16785-163, I. R. Iran}\\
		{\small\em$^d$School of Mathematics and Statistics, Beijing Jiaotong University, Beijing 100044, P. R. China}\\
		{\small\em$^e$ UP FAMNIT, University of Primorska, Koper 6000, Slovenia}\\
	}
	\date{}
	\maketitle
	
	\begin{abstract} In this paper, we introduce the concept of $k$-integral graphs. A graph $\Gamma$ is called $k$-integral if the extension degree of the splitting field of the characteristic polynomial of $\Gamma$ over rational field $\mathbb Q$ is equal to $k$. We prove that  the set of all finite connected graphs with given algebraic degree and maximum degree is finite. $1$-integral graphs are just integral ones, graphs all of whose eigenvalues are integer. We study
		$2$-integral Cayley graphs over finite groups $G$ with respect to  Cayley sets which are a union of conjugacy classes of $G$. Among other general results, we completely characterize all finite abelian groups having a connected $2$-integral Cayley graph with valency $2,3,4$ and $5$.
		Furthermore, we classify finite groups $G$ for which all Cayley graphs over $G$ with bounded valency are $2$-integral.	
	\end{abstract}	
	

	\section{Introduction and results}
	
	A graph is called {\it integral} if the eigenvalues of its adjacency matrix are all integers. 
	This concept was first introduced by Harary and Schwenk \cite{Which Graphs Have Integral Spectra} and they raised the question: which graphs are integral?
	After this, classification and construction of the integral graph have become an active topic. 
	A general approach to studying integral graphs is to focus on some special classes of graphs, such as Cayley graphs.
	A \textit{Cayley graph} over a group $G$ with respect to a subset $S$ of $G$, denoted by $\mathrm{Cay}(G,S)$, is a (di)graph with vertex set $G$ such that $(g,h)$ is an arc if and only if $hg^{-1}\in S$. If $S=\varnothing$ then $\mathrm{Cay}(G,S)$ is the empty graph, and if $1_G\in S$ then the corresponding Cayley graph has a loop at each vertex. 
	 Although in this paper the connection set $S$ is not necessarily inverse-closed, if $S$ is inverse-closed, i.e, $S=S^{-1}=\{s^{-1}\mid s\in S\}$,  then $\mathrm{Cay}(G,S)$ is an undirected graph. Furthermore, $\mathrm{Cay}(G,S)$ is connected  if and only if $G=\langle S\rangle$. 
	The Cayley graph over a cyclic group is also called {\it circulant}.
	A subset $S$ of a group $G$ is called \textit{normal} if for any $g\in G$, we have $g^{-1}Sg=S$. Clearly, a subset of any abelian group is a normal set. 
	The Cayley graph over a group $G$ with respect to a normal subset $S$ is  called a {\it normal Cayley graph}.
	A large number of results on the eigenvalues of Cayley graphs have been produced over the past more than four decades; for a survey on this topic see \cite{LZ}.
	All graphs in this paper are loop-free, non-empty, and are not necessarily undirected, but whenever a graph is, we will mention it.
	
	M\"{o}nius, Steuding and Stumpf \cite{Which graphs have non-integral spectra} introduced the concept  of the splitting fileds and the algebraic degrees of graphs to investigate which graph properties prevent integral eigenvalues.
	For a graph $\Gamma$, its {\it splitting field} $\mathbb{SF}(\Gamma)$ is the smallest field extension of the rational number field $\mathbb{Q}$ which contains all the eigenvalues of the adjacency matrix of $\Gamma$.
	The extension degree $[\mathbb{SF}(\Gamma):\mathbb{Q}]$ is called the {\it algebraic degree} of $\Gamma$, denoted by $\mathrm{deg}(\Gamma)$.
	Since eigenvalues of graphs are all algebraic integers, a graph $\Gamma$ is integral if and only if $\mathrm{deg}(\Gamma)=1$.
	Some works have been done, very recently, on determining algebraic degrees of Cayley graphs and their generalization \cite{Splitting fields of mixed cayley graphs over abelian groups,LM,Algebraic degree of spectra of Cayley hypergraphs,SMT-2, WGYF,WYF}.
	Note that the question raised by Harary and Schwenk can be translated as: which graphs have algebraic degree 1? 
	It is natural to ask which graphs have algebraic degree $k$ for a given positive integer $k$?
	To study this question, we define $k$-integral graphs as follows.
	
	\medskip
	\noindent \textbf{Definition.} For a given  positive integer $k$, a  graph $\Gamma$ is called {\it $k$-integral} if $\mathrm{deg}(\Gamma)=k$.
	\medskip
	
	To study $k$-integral graphs, it is crucial to know that the number of such connected graphs are finite.	
	In Section 3, we will give a positive answer to this question.
	We prove that for any positive integers $k$ and $\Delta$, the set of all finite connected graphs with algebraic degree at most $k$ and maximum degree at most $\Delta$ is finite (see Theorem \ref{finite}).
	Our result further improves \cite[Theorem 2]{Cvet} which states that all regular, connected, integral graphs of a fixed degree are finite.

	Let $G=\langle a\rangle\cong\mathbb Z_n$ be a cyclic group of order $n$ and $S$ be an inverse-closed subset of $G$. 
	In 2005, W. So proved that a loop-free graph $\mathrm{Cay}(G,S)$ is integral if and only if $S$ is a union of $G_n(d)$s, where $d\neq n$ is a divisor of $n$ and $G_n(d)=\{a^k\mid (k,n)=d\}$ \cite[Theorem 7.1]{So}.
	Since then some authors tried to generalize So's result to finite abelian groups and non-abelian groups \cite{Alprin, AP,KS}.
	An interesting question is how to construct $k$-integral Cayley graphs over cyclic groups.
	For $G=\langle a\rangle\cong\mathbb Z_n$ and any divisor $d\neq n$ of $n$, if we put $[a^d]=\{a^k\mid (k,n)=d\}$, then by So's result, $\mathrm{Cay}(G,[a^d])$ is integral. 
	In Section 4, we show that using some special subsets of $[a^d]$, one can construct $k$-integral Cayley graphs over cyclic groups (see Proposition \ref{action}).
	As a corollary, we also give a method to construct a $p$-integral Cayley graph over cyclic group $\mathbb{Z}_n$, where $n\geq 5$ and $p$ is a prime divisor of $n$ (see Corollary \ref{p}).
	
	The problem of classifying all finite groups having a connected integral undirected Cayley graph of given valency is started by Abdollahi and Vatandoost \cite{AV}. They classified all such groups for valency $2$ and $3$ in \cite{AV} and partially for valency $4$ in \cite{AV-2}. Then, the problem completely solved, using computer, for valency $4$ in \cite{MW} by Minchenko and Wanless. We are also interested in classifying all finite abelian groups having a connected $2$-integral undiredted Cayley graph with small valency. To this purpose, we define the set $\mathcal{G}_k$ to be the set of all finite groups having a connected undirected $2$-integral Cayley graph with valency $k$, and using Corollary \ref{asli} and some other  general results, we completely characterize all finite abelian groups belong to $\mathcal{G}_k$ for $2\leq k\leq 5$ (see Theorems \ref{2}, \ref{3}, \ref{4}, and \ref{5}). As a corollary, we also give the classification of $2$-integral abelian Cayley graphs with valency $2,3,4$ and $5$.

		In \cite{EK}, Est\'{e}lyi and Kov\'{a}cs determined all finite groups $G$ for which all undirected graphs $\mathrm{Cay}(G, S)$ are integral when $4\leq |S|\leq k$ for each integer $k\geq 4$.
		Ma and Wang \cite{MWkais} characterized the finite groups each of whose cubic undirected Cayley graphs are integral.
		Inspired by the above works, it is interesting to classify the finite groups $G$ for which all undirected graphs $\mathrm{Cay}(G, S)$ are $2$-integral when $2\leq |S|\leq k$ for each integer $k\geq 2$.
		In Section 6, we focus on this problem.
		We show that there is no group $G$ for which all undirected graphs  $\mathrm{Cay}(G, S)$ are $2$-integral when $4\leq |S|$ (see Theorem \ref{thm:B4 empty}).
		Moreover, we completely classify the finite groups $G$ that all undirected graphs $\mathrm{Cay}(G, S)$ are $2$-integral when $2\leq |S|\leq k$ where $k=2$ and $3$ (see Theorems \ref{thm:B2} and \ref{thm:B3}).

	\section{Preliminaries and notations}
	
	In this paper, (di)graphs are finite, loop-free and without multiple edges. Also the groups are finite.  Our notations are standard and  mainly taken from \cite{HIK} and \cite{JL}, but for the reader's
	convenience we recall some of them as follows:
	
	\begin{itemize}
		\item $\langle g\rangle$ : the cyclic group generated by $g$.
		\item $\mathbb Z_n$ : the additive group of integers modulo $n$.
		\item $\mathbb Z_n^*$ : the multiplicative group of $\mathbb Z_n$.
		\item $\mathrm{Aut}(G)$ : the automorphism group of the group $G$.
		\item $\varphi$ : the Euler function.
		\item $\tau$ : the automorphism of an abelian group which maps every element to its inverse.
		\item $G_1\times G_2$ : the direct product of groups $G_1$ and $G_2$. 
		\item $[g]$ : the set $\{h\in G\mid \langle h\rangle=\langle g\rangle\}$, where $g\in G$.
		\item $Z(G)$ : the center of $G$.
		\item $g^H$ : $\{g^\sigma\mid \sigma\in H\}$, where $g\in G$ and $H\leq\mathrm{Aut}(G)$.
		\item $\mathrm{Irr}(G)$ : the set of all inequivalent $\mathbb C$-irreducible characters of group $G$.
		\item $K_n$ : the complete graph with $n$ vertices.
		\item $K_{n,m}$ : the complete bipartite graph with parts having $m$ and $n$ vertices. 
		\item $\Gamma_1\vee\Gamma_2$ : the join of graphs $\Gamma_1$ and $\Gamma_2$.
		\item $\Gamma_1 \square \Gamma_2$ : the Cartesian product of graphs $\Gamma_1$ and $\Gamma_2$.
		\item $\Gamma_1 \otimes \Gamma_2$ : the direct product of $\Gamma_1$ and $\Gamma_2$.
		\item $\Gamma_1 \boxtimes \Gamma_2$ : the strong product of $\Gamma_1$ and $\Gamma_2$.
		\item $\Gamma_1[ \Gamma_2]$ : the lexicographic product of $\Gamma_1$ and $\Gamma_2$.
		\item $\mathbb{SF}(\Gamma)$ : the splitting field of a graph $\Gamma$.
	\end{itemize}

	Let $\Gamma$ be a graph and $\mathbb{SF}(\Gamma)$ be the splitting field of $\Gamma$. 
	By the definition of a $k$-integral graph, $\Gamma$ being $2$-integral means that $[\mathbb{SF}(\Gamma):\mathbb{Q}]=2$. 
	It is known that $[\mathbb{SF}(\Gamma):\mathbb{Q}]=2$ implies that $\mathbb{SF}(\Gamma)=\mathbb{Q}(\alpha)$, where $\alpha\in\mathbb{SF}(\Gamma)\setminus \mathbb{Q}$ and $\alpha^2\in\mathbb Z$. 
	Next we will show that there exists a $2$-integral graph with $n$ vertices for any integer $n\geq 3$.
	
	Before the proof, we recall that the eigenvalues of the complete bipartite graph $K_{m,n}$ are $0$ with multiplicity $m+n-2$ and $\pm\sqrt{mn}$ with multiplicity $1$. 
	Hence if $mn$ is not a square integer, $\mathbb{SF}(K_{m,n})=\mathbb{Q}(\sqrt{mn})$ and so $K_{m,n}$ is $2$-integral.

	\begin{lem}
		For any integer $n\geq 3$, there exists an undirected $2$-integral graph with $n$ vertices.
	\end{lem}
	\begin{proof}
		The eigenvalues of $P_3$ are $0$ and $\pm \sqrt{2}$, and the eigenvalues of $P_4$ are $\frac{\pm 1\pm \sqrt{5}}{2}$.
		Thus $P_3$ and $P_4$ are both $2$-integral. So we may assume that $n\geq 5$. If $n$ is odd, then there exist distinct integers $n_1$ and $n_2$ such that $n=n_1+n_2$ and $\sqrt{n_1n_2}$ is not an integer. Now $K_{n_1,n_2}$ is $2$-integral with $n$ vertices. Let $n=2m$ be even and $m\geq 3$. In this case, all distinct eigenvalues of $K_{m-2,m+2}$ are $0$ and $\pm\sqrt{m^2-4}$.
		Clearly, $\pm\sqrt{m^2-4}\notin \mathbb{Z}$ and $K_{m-2,m+2}$ is $2$-integral. This completes the proof.
	\end{proof}
	
	Next we introduce a special subset of a group, which will be useful to study $k$-integral Cayley graphs.
	Let $G$ be a finite  group, and for any $g\in G$, define $[g]:=\{h\in G\mid \langle g\rangle=\langle h\rangle\}$. If $g$ and $h$ are two  elements of $G$, then $[g]=[h]$ or
	$[g]\cap [h]=\emptyset$. Hence the set $\Omega(G):=\{[g]\mid g\in G\}$ is a partition of $G$. 
	In the case that $G=\langle a\rangle\cong \mathbb Z_n$ is a cyclic group of order $n$ generated by $a$, for a divisor $d\neq n$ of $n$, we have 
	$[a^d]=\{a^k\mid 1\leq k\leq n-1,~(k,n)=d\}$ which is denoted by $G_n(d)$ in the literature \cite{So}.  
	
	\section{The number of $k$-integral undirected graphs}
	One of the main questions in studying $k$-integral graphs is whether the number of such connected graphs with fixed maximum degree are finite. In this section, we give a positive answer to this question.
	
	\begin{pro}\label{abdollahi}
		There exists a function $f:\mathbb{N}\times\mathbb{N}\rightarrow\mathbb{N}$ such that the number of distinct eigenvalues of any finite undirected graph with maximum degree $\Delta$ and algebraic degree $k$ is at most $f(k,\Delta)$.
	\end{pro}
	\begin{proof}
		Let $\Gamma$ be an undirected graph with algebraic degree $k$ and  $p(x)$ be the characteristic polynomial of the adjacency matrix of $\Gamma$. Then we have $p(x)=p_1(x)^{t_1}\cdots p_l(x)^{t_l}$, where $p_i(x)$s are monic polynomials on $\mathbb Q[x]$, irreducible and pairwise relatively prime and $t_i$'s are positive integers. Since $p_i(x)$s are irreducible and pairwise relatively prime in $\mathbb{Q}$, the roots of $p_i(x)$s are pairwise distinct and $p_i(x)$s has no repeated roots. 
		
		Let $p_i(x)=x^{k_i}+a_1x^{k_i-1}+\cdots+a_{k_i-1}x+a_{k_i}$, where $a_i$s are integers. 
		Since the roots of $p_i(x)$ are also the roots of $p(x)$, $k_i$ divides $k$. Furthermore, since the roots of $p_i(x)$ are eigenvalues of $\Gamma$, the absolute value of any root of $p_i(x)$ is less than or equal to  $\Delta(\Gamma)$, the maximum degree of vertices of $\Gamma$. Now by Vieta's formula and triangle inequality, for each $j$, we have  $$| a_j|\leq \binom{k_i}{j}\Delta(\Gamma)^{k_i-j}.$$
		This means that the number of such polynomials and so the number of distinct eigenvalues
		of $\Gamma$ is less than or equal to a function of $k$ and $\Delta(\Gamma)$, as desired.
	\end{proof}
	\begin{thm}\label{finite}
		For any positive integers $k$ and $\Delta$, the set of all finite connected undirected graphs with algebraic degree at most $k$ and maximum degree at most $\Delta$ is finite.
	\end{thm}
	\begin{proof}
		Since for any connected undirected graph with $s$ distinct eigenvalues and diameter $D$, we have $D\leq s-1$, the result directly follows from Proposition \ref{abdollahi} and the fact that the order of a connected graph with given diameter and maximum degree is bounded.
	\end{proof}
	
	\section{$k$-integral Cayley graphs}
	In this section, we are going to study the algebraic degrees of normal Cayley graphs. 
	First let us recall two known results about the algebraic degrees of abelian Cayley graphs and  normal Cayley (di)graphs. Note that in these two results, $S$ is not necessarily inverse-closed and so $\Gamma$ can be a Cayley digraph. Furthermore, by $S^k$ we mean the multiset $\{s^k\mid s\in S\}$. 
	
	\begin{lem}\cite[Theorem 1]{LM}\label{lem:abelian-deg}
		Let $G$ be an abelian group of order $n$ and $\Gamma=\mathrm{Cay}(G,S)$ for some subset $S$ of $G$. Then the algebraic degree of $\Gamma$ is $\mathrm{deg}(\Gamma)=\frac{\varphi(n)}{|H|}$,
		where $H=\{k\in \mathbb{Z}_n^*\mid S^k=S\}$.	
	\end{lem}
	\begin{lem}\cite[Corollary 3.10]{SMT-2}\label{lem:normal-deg}
		Let $G$ be a finite group with exponent $m$, that is the smallest positive integer such that $g^{m}=1$ for all $g\in G$, and $\Gamma=\mathrm{Cay}(G,S)$ for some normal subset $S$ of $G$. Then the algebraic degree of $\Gamma$ is
		$\mathrm{deg}(\Gamma)=\frac{\varphi(m)}{|H'|}$,
		where $H'=\{k\in \mathbb{Z}_m^*\mid S^k=S\}$.
	\end{lem}

	By the following two lemmas, we reduce the study of $k$-integral normal Cayley graphs to the study of $k$-integral $\mathrm{Cay}(\langle g\rangle,S)$, where $\emptyset\neq S\subseteq [g]$.
	\begin{lem}\label{comp}
		Let $G$ be a finite group, $g\in G$ and $\emptyset\neq S\subseteq [g]$. Then $\langle S\rangle=\langle g\rangle$ and the splitting field of $\mathrm{Cay}(G,S)$ and $\mathrm{Cay}(\langle g\rangle,S)$ are the same.
	\end{lem}
	\begin{proof}
		Since $x\in [g]$ if and only if $\langle x\rangle=\langle g\rangle$, the first part is clear. On the other hand, $\mathrm{Cay}(G,S)$ is isomorphic to the disjoint union of $|G:\langle S\rangle|$ copies of
		$\mathrm{Cay}(\langle S\rangle,S)$, where $\mathrm{Cay}(\langle S\rangle,S)$ is a connected component of $\mathrm{Cay}(G,S)$. Hence the sets of all distinct eigenvalues of $\mathrm{Cay}(G,S)$ and $\mathrm{Cay}(\langle S\rangle,S)$ are the same and so their splitting fields are, as desired.
	\end{proof}
	
	\begin{lem}\label{union}
		Let $G$ be a finite group, $\Omega(G)=\{[g_1],\ldots,[g_t]\}$, $\varnothing\neq S$ be a normal  subset of $G$, and $\Gamma=\mathrm{Cay}(G,S)$. Then $S=\bigcup_{i=1}^t S_i$, where $S_i=[g_i]\cap S$ for each $1\leq i\leq t$, and for $S_i\neq \varnothing$, $\deg(\Gamma_i)$ is a divisor of $\deg(\Gamma)$, where $\Gamma_i=\mathrm{Cay}(G,S_i)$.
	\end{lem}
	\begin{proof}
		Since $\Omega(G)$ is a partition of $G$, $X=\{S_i\mid 1\leq i\leq t, S_i\neq\emptyset\}$ is a partition of $S$.
		Let $m$ be the exponent of $G$ and $m_i$ be the order of $g_i$. Then $m_i|m$.
		By Lemma \ref{lem:normal-deg}, $\deg(\Gamma)=\frac{\varphi(m)}{|H|}$, where $H=\{k\in\mathbb Z_m^*\mid S^k=S\}$. If $S_i\neq \varnothing$, then by Lemmas \ref{comp} and \ref{lem:abelian-deg}, $\deg(\Gamma_i)=\deg(\Cay(\langle g_i\rangle,S_i))=\frac{\varphi(m_i)}{|H_i|}$, where $H_i=\{k\in\mathbb Z_{m_i}^*\mid S_i^k=S_i\}$.

		Suppose $k\in H$. Then $(k,m)=1$ and $S^k=S$. Since $m_i|m$, clearly $(k,m_i)=1$. So $\langle g_i^k \rangle=\langle g_i\rangle$ and therefore $[g_i]^k=[g_i^k]=[g_i]$. Thus
	\[S_i^k=([g_i]\cap S)^k\subseteq [g_i]^k\cap S^k=[g_i]\cap S=S_i.\]
	 Since $X$ is a partition of $S$, we have $S_i^k=S_i$.
	 
	  Since $m_i|m$, the map $\pi:\Bbb Z_m^*\rightarrow\Bbb Z_{m_i}^*$ which maps any $k$(mod $m$) to $k$(mod $m_i$) is an onto group homomorphism. Let $K$ be the kernel of $\pi$. Then $|K|=\frac{\varphi(m)}{\varphi(m_i)}$, by the First Isomorphism Theorem. Let $\pi'$ be the restriction of $\pi$ to $H$. Then, by the above argument, $\pi'(H)\leq H_i$. Hence
	  $|H|$ divides $|K'||H_i|$, where $K'$ is the kernel of $\pi'$, again by the First Isomorphism Theorem. On the other hand, $K'=K\cap H$ and so $|K'|$ divides $|K|$. This implies that 
	  $|H|$ divides $\frac{\varphi(m)}{\varphi(m_i)}|H_i|$ and so $\frac{\varphi(m_i)}{|H_i|}$ divides $\frac{\varphi(m)}{|H|}$ as desired.
	\end{proof}

	Let $G=\langle a\rangle\cong\mathbb Z_n$ and $\Gamma=\mathrm{Cay}(G,S)$. 
	By Lemma \ref{lem:abelian-deg}, we know that $\deg(\Gamma)$ is a divisor of $\varphi(n)$. 
	In what follows, we are going to give a way to find possible inverse-closed generating sets of $G$ such that $\deg(\mathrm{Cay}(G,S))=1$ or $p$ for some prime $p\mid \varphi(n)$.
	\begin{pro}\label{action}
		Let $n\geq 3$,  $d\neq n$, $d\mid n$, $G=\langle a\rangle\cong\mathbb Z_n$ and $A=\mathrm{Aut}(G)$. Let $T_d=\{\sigma\in A\mid g^\sigma=g,~\forall g\in[a^d]\}$ and $K_d$ be a subgroup of $A$ containing $T_d$.
		For a fixed integer $1\leq k\leq n-1$ with $(k,n)=d$, put $S_{k,d}=\{(a^k)^\sigma\mid \sigma\in K_d\}$ and $\Gamma_{k,d}=\mathrm{Cay}(G,S_{k,d})$. Then
		\begin{itemize}
			\item[$(1)$] $S_{k,d}$ is inverse-closed if $K_d$ contains the element $\tau$ of $A$ which maps $a$ to $a^{-1}$.
			\item[$(2)$] $|S_{k,d}|=|K_d:T_d|\leq \varphi(\frac{n}{d})$,
			\item[$(3)$] $S_{k,d}$ is a generating set of $G$ if and only if $d=1$. In this case, $T_1=\{1_A\}$.
			\item[$(4)$] $\deg(\Gamma_{k,d})$ is a divisor of $|A:K_d|$.
			\item[$(5)$] $\Gamma_{k,d}$ is integral if and only if $A=K_d$. In this case, $S_{k,d}=[a^d]$. 
		\end{itemize} 
		In particular, if $|A:K_d|=p$ for some prime $p$, then $\Gamma_{k,d}$ is $p$-integral.
	\end{pro}
	\begin{proof}
		Part $(1)$ is obvious by the definition of $\tau$. 
		Note that $[a^d]$ is the set of all elements of $G$ with order $\frac{n}{d}$.
		Since the order of $a^k$ is $\frac{n}{(k,n)}=\frac{n}{d}$ and each automorphism of $G$ preserves the order of elements of $G$, $S_{k,d}\subseteq [a^d]$.
		Moreover, $A$ can acts on $[a^d]$ and $T_d$ is the kernel of this action.
		Hence $A/T_d$ is a permutation group on $[a^d]$. 
		Note that $S_{k,d}$ is the orbit of $a^k$ under the action of $K_d$ on $G$. 
		Since $A$ is abelian, $K_d/T_d$ is abelian. This implies that $K_d/T_d$ is regular on $S_{k,d}$, which means $|S_{k,d}|=|K_d:T_d|\leq |[a^d]|=\varphi(\frac{n}{d})$. This proves $(2)$. 
		Since $S_{k,d}\subseteq [a^d]$, $\langle S_{k,d}\rangle \subseteq\langle a^d\rangle$. Hence
		$S_{k,d}$ generates $G$ if and only if $d=1$. Clearly if $d=1$ then $T_1$ fixes $a$ and so $T_1=\{1_A\}$. This proves $(3)$.
		
		By Lemma \ref{lem:abelian-deg}, $\deg(\Gamma_{k,d})=\frac{|A|}{|H|}$, where $H=\{\sigma\in A\mid (S_{k,d})^\sigma=S_{k,d}\}$. Clearly, we have $T_d\leq H$. On the other hand $K_d\leq H$. Hence we have $K_d/T_d\leq H/T_d\leq A/T_d$ and so $|A:K_d|=|A/T_d:K_d/T_d|=|A/T_d:H/T_d||H/T_d:K_d/T_d|=|A:H||H/T_d:K_d/T_d|$ which implies that $\deg(\Gamma_{k,d})$ divides $|A:K_d|$, which proves $(4)$.
		
		If $K_d=A$ then $\deg(\Gamma_{k,d})=1$, by $(4)$, which means $\Gamma_{k,d}$  is integral. Conversely, suppose that $\Gamma_{k,d}$ is integral. Since $S_{k,d}$ is a subset of $[a^d]$, \cite[Theorem 7.1]{So} implies $S_{k,d}=[a^d]$. Now $A$ acts transitively on $[a^d]$ and so $A/T_d$ is a regular permutation group on $[a^d]$. Thus $|A/T_d|=|K_d/T_d|$, which means
		$A=K_d$. This completes the proof.
	\end{proof}
	
	As an interesting application of Proposition \ref{action}, one can construct Cayley graphs over cyclic groups with prime algebraic degrees.
	More precisely, for any integer $n\geq 5$ and $n\neq 6$, one can construct a $p$-integral Cayley graph over $\mathbb Z_n$, where $p$ is a prime divisor of $\varphi(n)$. 

	\begin{core}\label{p}
		Let $n\geq 5$ be an integer.
		\begin{itemize}
			\item[$(1)$] If $\varphi(n)$ is not a power of $2$, then for every odd prime divisor of $\varphi(n)$, there exists an undirected $p$-integral Cayley graph over $\mathbb Z_n$ with valency $\frac{\varphi(n)}{p}$;
			\item[$(2)$] If $\varphi(n)$ is a power of $2$, then there exists an undirected $2$-integral Cayley graph over $\mathbb Z_n$ with valency $\frac{\varphi(n)}{2}$.
		\end{itemize}
		In particular, for every prime $p$ and integer $k\geq 3$, there exists a connected $p$-integral undirected circulant graph of order $p^k$ and valency $p^{k-2}(p-1)$.
	\end{core}
	\begin{proof}
		Let $A=\mathrm{Aut}(\mathbb Z_n)$ and $\tau$ be the element of $A$ which maps each element to its inverse. 
		Following the notation in Proposition \ref{action} and letting $d=k=1$, we have
		$T_d=T_1=\{1_A\}$.
		
		First suppose that $\varphi(n)$ is not a power of $2$ and $p$ is an odd prime divisor of $\varphi(n)$.
		Since $p\div\varphi(n)$ and $A$ is an abelian group of order $\varphi(n)$, there exists a subgroup $K$ of $A$ such that $|A:K|=p$. 
		Let $K_{1}:=K$, $S_{1,1}:=\{a^\sigma\mid \sigma\in K\}$ and $\Gamma_{1,1}=\mathrm{Cay}(G,S_{1,1})$. 
		Next we will show that $\tau\in K$ and so $S_{1,1}$ is an inverse-closed generating set of $G$ by $(1)$ and $(3)$ of Proposition \ref{action}.
		Put $H:=\langle K,\tau\rangle$.
		If $\tau\notin K$ then $K$ is a proper subgroup of $H$ and $|H:K|=2$.
		Since $p=|A:K|=|A:H||H:K|=2|A:H|$, we have $p=2$ and $A=H$, a contradiction.
		Hence $\tau\in K$.
		Moreover, by Proposition \ref{action} $(2)$, $|S_{1,1}|=|K_1:T_1|=|K|=\frac{\varphi(n)}{p}$ and $\deg(\Gamma_{1,1})=p$. This proves $(1)$.
		
		Next suppose $\varphi(n)$ is a power of $2$, that is, $\varphi(n)=2^k$ for some $k\geq 3$.
		Hence $A$ is an abelian $2$-group of order $2^k$. We know that there exists a maximal subgroup $M\neq 1$ of $A$ containing $\langle \tau\rangle$. Since $A$ is abelian, $|A:M|$ must be a prime, which means $|A:M|=2$ and the second part follows from Proposition \ref{action} as above. 
		
		For every prime $p$ and integer $k\geq 3$, by $(1)$, $(2)$ and the fact $\varphi(p^k)=p^{k-1}(p-1)$, there exists a connected $p$-integral circulant graph of order $p^k$ and valency $p^{k-2}(p-1)$. This completes the proof.
	\end{proof}

		Suppose $G=\langle a\rangle\cong \mathbb{Z}_n$ and keep the notations in Proposition \ref{action}.
		For a prime divisor $p$ of $\varphi(n)$, to construct a $p$-integral Cayley graph over $G$, it suffices to construct a subgroup $K_d$ of $A$ for any given $d\mid n$ such that $|A:K_d|=p$ and $\langle \tau, T_d\rangle\leq K_d$.
		Next we are going to give the element $\tau$ of $A$ which maps $a$ to $a^{-1}$ first.
		Then we give a concrete example to illustrate how to construct the subgroup $K_d$ of $A$ containing $\langle \tau,T_d\rangle$ and $p$-integral Cayley graph over $G$.

		\begin{example}
			Let $G=\langle a\rangle\cong\mathbb Z_{20}$. 
			Then $G=P_1\times P_2$, where $P_1=\langle a^{5}\rangle\cong\mathbb Z_4$ and $P_2=\langle a^4\rangle\cong\mathbb Z_{5}$. 
			Let $A=\mathrm{Aut}(G)$, $A_1=\mathrm{Aut}(P_1)$ and $A_2=\mathrm{Aut}(P_2)$. 
			Then $A_1=\langle \pi_1\rangle\cong\mathbb Z_2$, $A_2=\langle \pi_2\rangle\cong\mathbb Z_{4}$ and $A=A_1\times A_2=\langle \pi_1\rangle\times \langle \pi_2\rangle=\{\pi_2^j, \pi_1\pi_2^j\mid 0\leq j\leq 3\}$, where
			\begin{align*}
				\pi_1:&~P_1\mapsto P_1,~~a^{5}\mapsto a^{15},\\		
				\pi_2:&~P_2\rightarrow P_2,~~~a^4\mapsto a^{12}.
			\end{align*}
			We have $\tau=\pi_1\pi_2^2$.
			For convenience, let $\sigma_j=\pi_2^j$ and $\theta_j=\pi_1\pi_2^j$ for each $0\leq j\leq 3$.
			Then $\tau=\theta_{2}$.
			Moreover, by $a=(a^{5})^{-3}(a^4)^{4}$, we have the following for each $0\leq j\leq 3$,
			\begin{align*}
				\sigma_j:~&G\rightarrow G,~~a\mapsto a^{-15+16\times3^j},\\
				\theta_j:~&G\rightarrow G,~~a\mapsto a^{-5+16\times3^j}.
			\end{align*}
			
			In order to construct an undirected $2$-integral Cayley graph over $G$, it is enough to construct a subgroup $K_d$ of $A$ with index $2$ containing $\langle \tau,T_d\rangle$ for any given $d|n$ by Proposition \ref{action}.
			Clearly, $d\mid n$ if and only if $d=1,2,4,5,10,20$.

			First let $d=1$. Then $T_d=\{1_A\}$. 
			Put $K_1:=\langle \tau,\sigma_2\rangle$. 
			Since $\tau\notin\langle\sigma_2\rangle$, $K_1=\langle\sigma_2\rangle\times\langle\tau\rangle$.
			Note that $|\langle\sigma_2\rangle|=o(\pi_2^2)=2$.
			Then $K_1=\{1,\tau,\sigma_2,\sigma_2\tau\}$, and so $|A:K_1|=2$. 
			Now let
			$1\leq k\leq 19$ be an integer comprime to $20$. Then, by the notations of Proposition \ref{action}, 
			\begin{eqnarray*}
				S_{k,1}=\{a^k,(a^k)^{\sigma_{2}},a^{-k},(a^{-k})^{\sigma_{2}}\}=\{a^k,a^{9k}, a^{-k}, a^{-9k}\}.
			\end{eqnarray*}
			Furthermore, $\mathrm{Cay}(G,S_{k,1})$ is a connected $4$-regular $2$-integral Cayley graph, where
			$S_{1,1}=S_{9,1}=S_{11,1}=S_{19,1}=\{a,a^{19},a^9,a^{11}\}$, $S_{3,1}=S_{7,1}=S_{13,1}=S_{17,1}=\{a^3,a^{17}, a^7,a^{13}\}$.
			
			Now let $d=2$. Then for $1\leq l\leq 20$, we have $(l,20)=2$ if and only if $l=2,6,14,18$. 
			By the above discussion and an easy calculation, we have $T_2=\{1_A,\theta_0\}$. 
			Put $K_2:=\langle\tau,\theta_0\rangle$. Then $|K_2|=4$, $|A:K_2|=2$, $|K_2:T_2|=2$, and
			$\mathrm{Cay}(G,S_{k,2})$, where $k=2,6,14,18$, is a $2$-regular and $2$-integral Cayley graph, where $S_{2,2}=S_{18,2}=\{a^2,a^{18}\}$, $S_{6,2}=S_{14,2}=\{a^6,a^{14}\}$.
			
			Now let $d=4$. Then, since $a^{20}=1$ and no power of $\pi_2$ maps $a^{16}$ to itself, similar to the previous paragraph, we see that $T_4=\{\sigma_0,\theta_0\}$. Hence, similar to the case $d=2$, we conclude that $\mathrm{Cay}(G,S_{k,4})$, where $k=4,8,12,16$ is $2$-integral. More precisely, $S_{4,4}=S_{16,4}=\{a^4,a^{16}\}$ and $S_{8,4}=S_{12,4}=\{a^8,a^{12}\}$.
			
			Now let $d=5$. Since $a^{25}=a^5$ and $a^{20}=1$, similar to the above, we see that $T_5=\{\sigma_0,\sigma_1,\sigma_2,\sigma_3\}$, which means that $\langle\tau,T_5\rangle=A$ and so, by Proposition \ref{action}, the corresponding Cayley graphs are integral. So in this case we can not construct any $2$-integral Cayley graph over $G$.
			
			In the case, $d=10,20$, clearly $T_{d}=A$, which again we can not construct any $2$-integral Cayley graph over $G$. Hence we have shown that if $S$ is one of the following sets, then $\mathrm{Cay}(G,S)$, where $G=\langle a\rangle\cong\mathbb Z_{20}$ is $2$-integral:
			$$\{a,a^{19},a^9,a^{11}\},~\{a^3,a^{17},a^{7},a^{13}\},~\{a^2,a^{18}\},~\{a^6,a^{14}\}, \{a^4,a^{16}\},~\{a^8,a^{12}\}.$$
		\end{example}

	Another important application of Proposition \ref{action} is that it enables us to study the $2$-integral Cayley graphs.
	First we give a characterization of the $2$-integrality of $\mathrm{Cay}(G,S)$, where $\emptyset\neq S\subseteq [g]$ for some $g\in G$.
	\begin{lem}\label{half}
		Let $G$ be a finite  group, $g\in G$ be an element of order $n$, $\varnothing\neq S\subseteq [g]$, maybe not inverse-closed,  and $\Gamma=\mathrm{Cay}(G,S)$. 
		Then $\Gamma$ is $2$-integral if and only if there exists a subgroup $H$ of $\mathrm{Aut}(\langle g\rangle)$ such that $|\mathrm{Aut}(\langle g\rangle):H|=2$, $S=s^H$ for all $s\in S$ and $|S|=\frac{\varphi(n)}{2}$.
	\end{lem}
	\begin{proof}
		Suppose $\Gamma$ is $2$-integral. Then $\mathrm{Cay}(\langle g\rangle, S)$ is $2$-integral by Lemma \ref{comp} and $S\neq [g]$ by \cite[Corollay 7.2]{So}. 
		Furthermore, by Lemma \ref{lem:abelian-deg}, $|A:H|=2$, where $A=\mathrm{Aut}(\langle g\rangle)\cong\mathbb Z_n^*$ and $H=\{\sigma\in A\mid S^\sigma=S\}$.
		Since for any $\sigma\in A$ there exists an integer $1\leq i\leq n$ with $(i,n)=1$ such that $g^\sigma=g^i$, $A$ acts transitively on $[g]$. 
		Note that $|A|=|[g]|$, which means
		that $A$ acts regularly on $[g]$. Hence the action of $H$ on each orbit of $H$ is regular. Let $s\in S$. Then $|s^H|=|H|=\frac{|A|}{2}=\frac{\varphi(n)}{2}$. Since $H$ acts semiregularly on $S$, we have $|H|$ divides $|S|$. On the other hand, $s^H\subseteq S\subset [g]$ implies $\frac{\varphi(n)}{2}=|H|=|s^H|\leq |S|<|[g]|=\varphi(n)$ and so $S=s^H$.
		
		Conversely, suppose that $S=s^H$ for some $s\in S$, where $H$ is a subgroup of $\mathrm{Aut}(\langle g\rangle)$ such that $|\mathrm{Aut}(\langle g\rangle):H|=2$ and $|S|=\frac{\varphi(n)}{2}$. 
		Let $\Sigma=\mathrm{Cay}(\langle g\rangle, S)$. Then, by Proposition \ref{action} (putting $d=1$), $\deg(\Sigma)=2$. Now Lemma \ref{comp} implies that $\Gamma$ is $2$-integral as desired.
	\end{proof}
	
	By Lemma \ref{half}, we can characterize the generating sets of  $2$-integral normal Cayley graphs.
	
	\begin{core}\label{asli}
		Let $G$ be a finite group, $\Omega(G)=\{[g_1],\ldots,[g_n]\}$, $S$ is a normal subset of $G$ and $\Gamma=\mathrm{Cay}(G,S)$.
		Then $S=\bigcup_{i=1}^{n}S_i$, where $S_i=[g_i]\cap S$ for each $1\leq i\leq n$.
		If $\Gamma$ is $2$-integral, then for $S_i\neq \varnothing$,
		\begin{itemize}
			\item [$(1)$]  there exists $s_i\in S$ such that $S_i=[s_i]$ or $S_i=s_i^H\subset [s_i]$ for some subgroup $H$ of $\mathrm{Aut}(\langle s_i\rangle)$ of index $2$, and $|S_i|=\frac{\varphi(n_i)}{2}$, where $n_i=o(s_i)$;
			\item[$(2)$] there exists $1\leq i_0\leq n$ such that $S_{i_0}$ is of the latter form. 
		\end{itemize}
	\end{core}
	\begin{proof}
		Suppose $\Gamma$ is $2$-integral. 
		By Lemma \ref{union}, $S=\bigcup_{i=1}^{n}S_i$ is a disjoint union of sets $S_1,\cdots,S_n$, where $S_i=[g_i]\cap S$ for each $1\leq i\leq n$, and if $S_i\neq \varnothing$, $\deg(\Gamma_i)=1$ or $2$, where $\Gamma_i=\mathrm{Cay}(G,S_i)$. 
		For $S_i\neq \varnothing$, $\deg(\Gamma_i)=1$ if and only if $S_i=[g_i]$ by \cite[Corollary 7.2]{AP}, and $\deg(\Gamma_i)=2$ if and only if  there exists
		a subgroup $H$ of $\mathrm{Aut}(\langle g_i\rangle)$ such that $|\mathrm{Aut}(\langle g_i\rangle):H|=2$, $S_i=s^H$ for all $s\in S_i$ and $|S_i|=\frac{\varphi(o(g_i))}{2}$ by Lemma \ref{half}. 
		If for each $i$ with $S_i\neq \varnothing$, we have	$\deg(\Gamma_i)=1$, then $S$ is a union of $[g_i]$s and so $\deg(\Gamma)=1$, by \cite[Proposition 4.1]{AP}, which contradicts  the $2$-integrality of $\Gamma$. 
		Hence there exists $i_0$ such
		that $\deg(\Gamma_{i_0})=2$. 
	\end{proof}
	

	\section{Finite abelian groups admitting a connected $2$-integral undirected Cayley graph with  small valency}
	Recall that a cyclic group $G=\langle a\rangle$ of order $n$ admits a connected integral undirected Cayley graph of valency $2$ if and only if $n=3,4,6$ \cite[Lemma 2.7]{AV}. Also recall that an undirected Cayley graph $\mathrm{Cay}(G,S)$ over
	an abelian group $G$ is integral if and only if $S$ is a union of some $[g]$s, where $g\in S$ \cite{Alprin,AP}.  If $X$ be a non-empty subset of a group $G$ and $\chi$ be a character of $G$, we set $\chi(X)=\sum_{x\in X}\chi(x)$. It is well-konw that if $S$ be a conjugate-closed subset of a group $G$ and $\Gamma=\Cay(G,S)$, then eigenvalues of $\Gamma$ are $\frac{\chi(S)}{\chi(1)}$ with multiplicity $\chi(1)^2$, where $\chi$ runs over $\mathrm{Irr}(G)$,  see \cite[Theorem 1]{Z} or \cite{Ito}. In this section, we use this fact frequently.
	
	Let $\mathcal{G}_k$ be the set of all finite groups admitting a connected $2$-integral undirected Cayley graph  with valency $k$. In this section, we completely characterize all finite abelian groups belong to $\mathcal{G}_k$ for $2\leq k\leq 5$.
	
	\subsection{$\mathcal{G}_2$}
	\begin{thm}\label{2}
		Let $G$ be a finite abelian group. Then $G\in\mathcal{G}_2$ if and only if $G\cong\mathbb Z_n$, where $n=5,8,10,12$.
		Furthermore, a cycle with order $n$ is $2$-integral if and only if $n=5,8,10,12$.
	\end{thm}
	\begin{proof}
		Let $S=\{x,y\}$ be an inverse-closed generating set of $G$ and $\Gamma=\mathrm{Cay}(G,S)$ is $2$-integral. Then either $x^2=y^2=1$ or $y=x^{-1}$. In the first case, $S=S_1\cup S_2$, where
		$S_1=\{x\}=[x]$, $S_2=\{y\}=[y]$, which means $\mathrm{Cay}(G,S_i)$, $i=1,2$, are both integral which contradicts Corollary \ref{asli}. In the later case, we have $S\subset [x]$ and $G=\langle x\rangle$. Let $n=o(x)$. Then, Lemma \ref{half} implies that $\varphi(n)=4$ and so $n=5,8,10,12$. 
		
		Conversely, suppose that $G=\langle x\rangle\cong\mathbb Z_n$, where $n=5,8,10,12$. Let $A=\mathrm{Aut}(G)$ and put $H=\langle \tau\rangle$. Then in each case $|A:H|=2$ and $|S|=\frac{\varphi(n)}{2}$, where $S=x^H$.  Hence $\mathrm{Cay}(G,S)$ is $2$-integral, by Lemma \ref{half}.
	\end{proof}
	
	\subsection{$\mathcal{G}_3$}
	To classify all finite abelian groups in $\mathcal{G}_3$, we need the following general result.
	\begin{lem}\label{+1}
		Let $G$ be a finite group and $S$ be a normal subset of $G$.  Then for any $1\neq x\in Z(G)$ in which $S\cap [x]=\varnothing$, $\mathrm{Cay}(G,S)$ and $\mathrm{Cay}(G,S\cup [x])$ have the same splitting field. In particular, if $x\in Z(G)\setminus S$ is an involution, then $\mathrm{Cay}(G,S)$ and $\mathrm{Cay}(G,S\cup \{x\})$ have the same splitting field.
	\end{lem}
	\begin{proof}
		Let $\chi\in\mathrm{Irr}(G)$ and $1\neq x\in Z(G)$ such that $S\cap [x]=\varnothing$. Since $x\in Z(G)$, \cite[Exercise 5 of Chapter 13]{JL} implies that for each $i$ we have 
		$\chi(x^i)=\lambda^i\chi(1)$, where $\lambda$ is an $n$th root of unity and $n=o(x)$. Hence $\frac{\chi([x])}{\chi(1)}=\sum_{1\leq i\leq n,\\ (i,n)=1}\lambda^i$ is an integer, say $t$, by \cite[Lemma 22.15]{JL}. Since $\frac{\chi(S\cup [x])}{\chi(1)}=\frac{\chi(S)+\chi([x])}{\chi(1)}=\frac{\chi(S)}{\chi(1)}+t$, clearly the splitting field of
		$\mathrm{Cay}(G,S)$ and $\mathrm{Cay}(G,S\cup [x])$ are the same, as desired.
	\end{proof}
	
	\begin{thm}\label{3}
		Let $G$ be a finite abelian group. Then $G\in\mathcal{G}_3$ if and only if $G$ is isomorphic to $\mathbb Z_n$ or $\mathbb Z_n\times\mathbb Z_2$, where $n=8,10,12$.
		Furthermore, a connected cubic abelian undirected Cayley graph $\Gamma$ is $2$-integral if and only if it is isomorphic to one of the  following $6$ Cayley graphs $\mathrm{Cay}(G,S)$, where:
		\begin{itemize}
			\item[$\romannumeral1$.] $G=\langle x\rangle \cong \mathbb{Z}_n$ where $n=8,10,12$ and $S=\{x,x^{-1},x^{\frac{n}{2}}\}$;
			\item[$\romannumeral2$.] $G=\langle x\rangle \times \langle y\rangle\cong \mathbb{Z}_n\times \mathbb{Z}_2$ where $n=8,10,12$ and $S=\{x,x^{-1},y\}$.
		\end{itemize}
	\end{thm}
	
	\begin{proof}
		Let $S=\{x,y,z\}$ be an inverse-closed generating set of $G$ and $\Gamma=\mathrm{Cay}(G,S)$ is $2$-integral. We deal with the following cases:
		
		\textbf{Case I.} $x^2=y^2=z^2$. In this case, $S=S_1\cup S_2\cup S_3$, where $S_1=[x]$, $S_2=[y]$ and $S_3=[y]$, which contradicts Corollary \ref{asli}.
		
		\textbf{Case II.} $y=x^{-1}$ and $z=x^{\frac{n}{2}}$, where $o(x)=n$ is even. Then $G=\langle x\rangle$ and $S=S_1\cup S_2$, where $S_1=\{x,x^{-1}\}\subseteq [x]$ and
		$S_2=[x^{\frac{n}{2}}]$. By Corollary \ref{asli}, $4=\varphi(n)$ which means $n=8,10,12$.
		Hence in this case $G=\langle x \rangle \cong \mathbb Z_n$, where $n=8,10,12$.
		
		\textbf{Case III.} $y=x^{-1}$, $z^2=1$ and $z\notin\langle x\rangle$. In this case, $G=\langle x\rangle\times\langle z\rangle\cong\mathbb Z_n\times\mathbb Z_2$, where $n=o(x)$. Furthermore, $S=S_1\cup S_2$, where $S_1=\{x,x^{-1}\}\subset [x]$ and $S_2=[z]$. Similar to the Case II, we get $\varphi(n)=4$ and so $n=5,8,10,12$.
		
		Hence we have showed that if $\Gamma=\mathrm{Cay}(G,S)$ is $2$-integral, then $G=\langle x\rangle \cong\mathbb Z_n$ and $S=\{x,x^{-1},x^{\frac{n}{2}}\}$ or $G=\langle x\rangle\times\langle z\rangle\cong\mathbb Z_n\times\mathbb Z_2$ and $S=\{x,x^{-1},z\}$, where $n=8,10,12$. This proves one direction. 
		
		For the converse direction, first we suppose that $G=\langle x\rangle \cong\mathbb Z_n$ and $S=S_1\cup S_2$, where $S_1=\{x,x^{-1}\}\subseteq [x]$, $S_2=\{x^{\frac{n}{2}}\}=[x^{\frac{n}{2}}]$ and $n=8,10,12$.
		Note that $S_1\cap [x^{\frac{n}{2}}]=\varnothing$, then $\mathbb{SF}(\mathrm{Cay}(G,S))=\mathbb{SF}(\mathrm{Cay}(G,S_1))$ by Lemma \ref{+1}.
		Use the same argument as in the proof of Theorem \ref{2}, we have $\mathrm{Cay}(G,S_1)$ is $2$-integral and so $\mathrm{Cay}(G,S)$ is $2$-integral. Hence $G\cong\mathbb Z_n\in \mathcal{G}_3$, where $n=8,10,12$.
		Similar arguments show the desired results for the case $G\cong\mathbb Z_n\times\mathbb Z_2$, where $n=8,10,12$. 
	\end{proof}
	
	\subsection{$\mathcal{G}_4$}
	To characterize all finite abelian groups belong to $\mathcal{G}_4$, we need the following general result:
	\begin{lem}\label{induced}
		Let $G$ be a finite group, $S_1,\ldots,S_k$ be normal subsets of $G$, $\langle S_i\rangle\cap\langle S_j\rangle=\{1\}$ for all distinct $i,j$, and $S=S_1\cup\cdots\cup S_k$. Let $\Gamma=\mathrm{Cay}(G,S)$ and  $\Gamma_i=\mathrm{Cay}(G,S_i)$, $i=1,\ldots,k$. Then
		\begin{itemize}
			\item[$(1)$] $\mathbb{SF}(\Gamma_i)\subseteq\mathbb{SF}(\Gamma)$ for all $i$,
			\item[$(2)$] $\deg(\Gamma_i)$ divides $\deg(\Gamma)$ for all $i$,
			\item[$(3)$] if $\deg(\Gamma_i)=\deg(\Gamma)=2$ for some $i$, then $\mathbb{SF}(\Gamma)=\mathbb{SF}(\Gamma_i)$,
			\item[$(4)$] if $\mathbb{SF}(\Gamma_i)=\mathbb{F}$ for all $i$, then $\mathbb{SF}(\Gamma)=\mathbb{F}$.
		\end{itemize}
	\end{lem}
	\begin{proof}
		Put $G_i:=\langle S_i\rangle$, $i=1,\ldots,k$. Since $S_i$ is a normal subset of $G$, $G_i$ is a normal subgroup of $G$. Let $\lambda$ be an eigenvalue
		of $\Gamma_i$ and $\mathrm{Irr}(G)=\{\chi_1,\cdots,\chi_k\}$. Then $\lambda=\frac{\psi(S_i)}{\psi(1)}$ for some $\psi\in\mathrm{Irr}(G_i)$. Let $\psi\uparrow G$ be the induced character from $\psi$. Then $\psi\uparrow G=d_1\chi_1+\cdots+d_k\chi_k$ for some integers $d_1,\ldots,d_k$. Since $S_i$s are normal subsets of $G$ and $\langle S_l\rangle\cap\langle S_{l'}\rangle=\{1\}$ for distinct $l,l'$, we have $(\psi\uparrow G)(S_i)=\frac{1}{|G_i|}\psi(S_i)$ and $(\psi\uparrow G)(S\setminus S_i)=0$. Hence $\frac{1}{|G_i|}\psi(S_i)=(\psi\uparrow G)(S)=d_1\chi_1(S)+\cdots+d_k\chi_k(S)$. Since the splitting fields contain the rational field and the character degrees are positive integers, we conclude that $\lambda\in\mathbb{SF}(\Gamma)$. Hence $\mathbb{SF}(\Gamma_i)\subseteq\mathbb{SF}(\Gamma)$. This proves $(1)$. Parts $(2)$ and $(3)$ are direct consequences of $(1)$.
		
		Let $\lambda$ be an eigenvalue of $\Gamma$. Since $S_i$s are normal subsets of $G$, $S$ is also a normal subset of $G$. Hence there exists $\chi\in\mathrm{Irr}(G)$ such that
		$\lambda=\frac{\chi(S)}{\chi(1)}$. On the other hand, $S_i$s are pairwise disjoint and so $\chi(S)=\chi(S_1)+\cdots+\chi(S_k)$. Since for each $i$, $\chi(S_i)\in\mathbb{F}$, we conclude
		that $\lambda\in\mathbb{F}$. Thus $\mathbb{SF}(\Gamma)\subseteq\mathbb{F}$. Now $(1)$ implies that $\mathbb{SF}(\Gamma)=\mathbb{F}$ as desired.
	\end{proof}
	\begin{thm}\label{4}
		Let $G$ be a finite abelian group. Then $G\in\mathcal{G}_4$ if and only if $G$ is isomorphic to one the following groups:
		\begin{itemize}
			\item[$(1)$] $\mathbb Z_n$, where $n=8,10,12,15,16,20,24,30$.
			\item[$(2)$] $\mathbb Z_n\times\mathbb Z_2^2$, where $n=5,8,10,12$.
			\item[$(3)$] $\mathbb Z_n\times\mathbb Z_2$, where $n=8,10,12$.
			\item[$(4)$] $\mathbb Z_n\times\mathbb Z_m$, where $(n,m)$ is one of the pairs $(3,12)$, $(4,8)$, $(4,10)$, $(4,12), (6,8)$, $(6,10)$, $(6,12)$, $(5,5)$, $(5,10)$, $(8,8)$, $(10,10)$, $(12,12)$.
		\end{itemize}
		Furthermore, a connected $4$-regular abelian undirected Cayley graph $\Gamma$ is $2$-integral if and only if it is isomorphic to one of the  following $39$ Cayley graphs $\mathrm{Cay}(G,S)$, where
		\begin{itemize}
			\item[$\romannumeral1$.] $G=\langle x\rangle \cong \mathbb{Z}_n$ where $n=15,16,20,24,30$, and $S=\{x,x^{-1},x^k,x^{-k}\}$ where $1\leq k\leq n$ such that $(k,n)=1$,
			\item[$\romannumeral2$.] $G=\langle x\rangle \cong \mathbb{Z}_8$ and $S=\{x,x^2,x^6,x^7\}$,		
			\item[$\romannumeral3$.] $G=\langle x\rangle \cong \mathbb{Z}_{10}$ and $S=\{x,x^2,x^8,x^9\}$,
			\item[$\romannumeral4$.] $G=\langle x\rangle \cong \mathbb{Z}_{12}$ and $S=\{x,x^2,x^{10},x^{11}\}$ or $\{x,x^3,x^9,x^{11}\}$ or $\{x,x^4,x^8,x^{11}\}$,
			\item[$\romannumeral5$.] $G=\langle x\rangle\times \langle w\rangle\cong \mathbb{Z}_n\times \mathbb{Z}_2$ where $n=8,10,12$ and $S=\{x,x^{-1},x^{\frac{n}{2}},w\}$,
			\item[$\romannumeral6$.] $G=\langle x\rangle\times \langle z\rangle\times \langle w\rangle\cong \mathbb{Z}_n\times \mathbb{Z}_2^2$ where $n=5,8,10,12$ and $S=\{x,x^{-1},z,w\}$,
			\item[$\romannumeral7$.] $G=\langle x\rangle\times \langle w\rangle\cong \mathbb{Z}_n\times \mathbb{Z}_m$ where $(n,m)$ is one of the pairs $(3,12)$, $(4,8)$, $(4,10)$, $(4,12), (6,8)$, $(6,10)$, $(6,12)$, $(5,5)$, $(5,10)$, $(8,8)$, $(10,10)$, $(12,12)$, and $S=\{x,x^{-1},y,y^{-1}\}$.
		\end{itemize}
	\end{thm}
	\begin{proof}
		Let $S=\{x,y,z,w\}$ be an inverse closed generating set of $G$ and $\mathrm{Cay}(G,S)$  is $2$-integral. Then we deal with the following cases:
		
		\textbf{Case I.} $x^2=y^2=z^2=w^2=1$. In this case, $G$ is an elementary abelian $2$-group isomorphic to $\mathbb Z_2^n$, where $n=3$ or $4$. Furthermore, $S=S_1\cup S_2\cup S_3\cup S_4$,
		where $S_1=\{x\}=[x]$, $S_2=\{y\}=[y]$, $S_3=\{z\}=[z]$ and $S_4=\{w\}=[w]$, which contradicts Corollary \ref{asli}.
		
		\textbf{Case II.} $z^2=w^2=1$, $z,w\notin\langle x\rangle$  and $y=x^{-1}$. In this case, $G=\langle x\rangle\times\langle z\rangle\times\langle w\rangle\cong\mathbb Z_n\times\mathbb{Z}_2^2$, where $n=o(x)$. Furthermore, $S=S_1\cup S_2\cup S_3$, where $S_1=\{x,x^{-1}\}\subseteq [x]$, $S_2=\{z\}=[z]$ and $S_3=\{w\}=[w]$. Now Corollary \ref{asli} implies
		that $S_1\neq [x]$ and $\varphi(n)=4$ which means $n=5,8,10,12$.
		
		\textbf{Case III.} $y=x^{-1}$, $z=x^{\frac{n}{2}}$, where $n=o(x)$ is even and $w^2=1$. In this case, $G\cong\mathbb Z_n\times\mathbb Z_2$. Also $S=S_1\cup S_2\cup S_3$, where
		$S_1=\{x,x^{-1}\}\subseteq [x]$, $S_2=\{x^\frac{n}{2}\}=[x^\frac{n}{2}]$ and $S_3=\{w\}=[w]$. Then by Corollary \ref{asli}, $S_1\neq [x]$ and $\varphi(n)=4$, which means $n=8,10,12$.
		
		\textbf{Case IV.} $y=x^k$ for some $k\neq 1,-1$, $z=x^{-1}$ and $w=y^{-1}$. In this case, $y^2\neq 1$ and $G=\langle x\rangle\cong\mathbb Z_n$, where $n=o(x)$. 
		
		First assume that $(k,n)=1$. Then $S\subset [x]$ and Lemma \ref{half} implies that $\varphi(n)=8$, which means $n=15,16,20,24,30$. 
		
		Now let $(k,n)=d\neq 1$. Then $S=S_1\cup S_2$, where $S_1=\{x,x^{-1}\}\subseteq [x]$ and $S_2=\{x^k,x^{-k}\}\subseteq [x^k]$. 
		If $S_1=[x]$ and $S_2\neq[x^k]$ then $n=3,4,6$ and $\frac{n}{d}=5,8,10,12$, respectively, which is impossible.
		If $S_1\neq [x]$ and $S_2=[x^k]$, then $n=5,8,10,12$ and $\frac{n}{d}=3,4,6$, respectively, which implies $(n,k)=(8,2)$, $(12,2)$, $(12,3)$ or $(12,4)$.
		If $S_1\neq [x]$ and $S_2\neq [x^k]$ then $n,\frac{n}{d}=5,8,10,12$, which implies $(n,k)=(10,2)$.

		\textbf{Case V.} $\langle x\rangle\cap\langle y\rangle=1$. Then $G=\langle x\rangle\times\langle y\rangle\cong\mathbb Z_n\times\mathbb Z_m$, where $n=o(x)$ and $m=o(y)$. We may assume that $n\leq m$. In this case, $S=S_1\cup S_2$, where $S_1=\{x,x^{-1}\}\subseteq [x]$ and $S_2=\{y,y^{-1}\}\subseteq [y]$. By Corollary \ref{asli}, $S_1=[x], S_2\neq [y]$ or
		$S_1\neq [x],S_2=[y]$ or $S_1\neq [x],S_2\neq [y]$. By a similar discussion to the above cases, in the first case $n=3,4,6,~m=5,8,10,12$, in the second case $n=5,8,10,12,~m=3,4,6$ and in the later case $n,m=5,8,10,12$.
		The second case is impossible because $n\leq m$. 
		Next we will show that in the last case $(n,m)$ must be $(5,5)$, $(5,10)$, $(8,8)$, $(10,10)$ or $(12,12)$.	
		Let $g\in G$ and $o(g)=k\geq 2$ and $C_k=\mathrm{Cay}(\langle g\rangle,\{g,g^{-1}\})$.
		By an easy computation, we have $\mathbb{SF}(C_{5})=\mathbb{SF}(C_{10})=\mathbb Q[\sqrt{5}]$, $\mathbb{SF}(C_8)=\mathbb Q[\sqrt{2}]$ and $\mathbb{SF}(C_{12})=\mathbb Q(\sqrt{3})$. Hence, by Corollary \ref{asli} and Lemma \ref{induced}, we have the result as desired. Moreover, Cases $\mathrm{ \uppercase\expandafter{\romannumeral1}}$-$\mathrm{\uppercase\expandafter{\romannumeral5} }$ prove one direction.
		
		The proof of converse direction of $G\cong\mathbb Z_n\times\mathbb{Z}_2^2$ where $n=5,8,10,12$, and $G\cong\mathbb Z_n\times\mathbb{Z}_2$ where $n=8,10,12$ are similar to the proof of converse direction of Theorem \ref{3}.
		
		Consider the converse direction of case $G=\langle x\rangle \cong \mathbb Z_n$, where $o(x)=n$ is shown in $(1)$.
		Let $\Sigma_{k}=\mathrm{Cay}(\langle x\rangle,\{x,x^k,x^{-1},x^{-k}\})$. 
		First suppose that $n=15,16,20,24,30$ and $(n,k)=1$. 
		Let $\sigma_k:x\mapsto x^k$ and $\tau:x\mapsto x^{-1}$.
		Then $H=\langle \tau,\sigma_k\rangle$ is a subgroup of index $2$ in $\mathrm{Aut}(G)$ and $x^H=\{x,x^k,x^{-1},x^{-k}\}$.
		By Lemmas \ref{half}, we have $\mathrm{SF}(\Sigma_{k})=2$.
		Now suppose that $n=8,10,12$.
		By a tedious computation, one can see that $\mathbb{SF}(\Sigma_{k})=\mathbb Q(\sqrt{2})$ if $(n,k)=(8,2)$, $(12,2)$ or $(12,4)$,
		$\mathbb{SF}(\Sigma_{k})=\mathbb Q(\sqrt{3})$ or $\mathbb Q(\sqrt{5})$ if $(n,k)=(12,3)$ or $(10,2)$ respectively.	
		
		Finally, consider the converse direction of case $G=\langle x\rangle \times \langle y\rangle \cong \mathbb Z_n\times\mathbb Z_m$, where $(n,m)$ shown in $(3)$.
		Let $S=S_1\cup S_2$ and $S_1=\{x,x^{-1}\}\subseteq [x]$ and $S_2=\{y,y^{-1}\}\subseteq [y]$.
		Note that $C_3,C_4,C_6$ are integral with valency $2$ and $\mathbb{Z}_{8},\mathbb{Z}_{10},\mathbb{Z}_{12}$ are all in $\mathcal{G}_2$ by Theorem \ref{2}.
		Hence by Lemma \ref{+1}, $\mathbb Z_n\times\mathbb Z_m \in \mathcal{G}_4$, where $(n,m)=(3,12)$, $(4,8)$, $(4,10)$, $(4,12)$, $(6,8)$, $(6,10)$ and $(6,12)$.
		On the other hand, since $\mathbb{SF}(C_{5})=\mathbb{SF}(C_{10})=\mathbb Q[\sqrt{5}]$, $\mathbb{SF}(C_8)=\mathbb Q[\sqrt{2}]$ and $\mathbb{SF}(C_{12})=\mathbb Q(\sqrt{3})$, 
		by Lemma \ref{induced}, $\mathbb Z_n\times\mathbb Z_m \in \mathcal{G}_4$, where $(n,m)=(5,5)$, $(5,10)$, $(8,8)$. 
		This completes the proof.
	\end{proof}
	\subsection{$\mathcal{G}_5$}
	To characterize all finite abelian groups in $\mathcal{G}_5$, we need the following general results.
	\begin{lem}\label{-1}
		Let $G$ be a finite group of even order, $S$ be a normal subset of $G$ and there exists an involution $x\in Z(G)\cap S$. If $\mathrm{Cay}(G,S)$ is connected with the splitting field $\mathbb{F}$, then 
		\begin{itemize}
			\item [$(1)$] $G=\langle S\setminus\{x\}\rangle$ or $G=\langle S\setminus\{x\}\rangle\times\langle x\rangle\cong \langle S\setminus\{x\}\rangle\times\mathbb{Z}_2$;
			\item [$(2)$] $\mathrm{Cay}(\langle S\setminus\{x\}\rangle, S\setminus\{x\})$ is a connected graph with the splitting filed $\mathbb{F}$. 
		\end{itemize}
	\end{lem}
	\begin{proof}
		Since $\mathrm{Cay}(G,S)$ is connected, $G=\langle S\rangle=\langle (S\setminus\{x\})\cup \{x\}\rangle$. 
		Since $x\in Z(G)$, $G=\langle S\setminus\{x\}\rangle \langle x\rangle$. 
		Moreover, $S\setminus\{x\}$ is also conjugate-closed, and since $x$ has order $2$, $G=\langle S\setminus\{x\}\rangle$ or $G=\langle S\setminus\{x\}\rangle\times\langle x\rangle\cong \langle S\setminus\{x\}\rangle\times\mathbb{Z}_2$. 
		This proves $(1)$.

		Let $\Gamma_1=\mathrm{Cay}(G, S\setminus\{x\})$. We claim that the splitting field of $\Gamma_1$ is $\mathbb{F}$. 
		First we suppose that $G=\langle S\setminus\{x\}\rangle$.
		Then Lemma \ref{+1} implies that $\mathrm{Cay}(G,S\setminus\{x\})$ and $\mathrm{Cay}(G,S)$ have the same splitting filed and so the splitting filed of $\Gamma_1$ is $\mathbb{F}$. 
		Next we suppose that $G=\langle S\setminus\{x\}\rangle\times\langle x\rangle$.
		Let $\lambda$ be an eigenvalue of $\mathrm{Cay}(G,S)$. 
		Then $\lambda=\frac{\chi(S)}{\chi(1)}$, for some $\chi\in\mathrm{Irr}(G)$. Since $x\in Z(G)$ we have $\chi(x)=\pm\chi(1)$, see \cite[Exrecise 5 of Chapter 13]{JL}, which implies $\lambda=\frac{\chi(S\setminus\{x\})}{\chi(1)}\pm 1$. Let $H=\langle S\setminus\{x\}\rangle$ and $\chi\mid_H$ be the restriction of $\chi$ to $H$. Since $|G:H|=2$ and $\chi(x)\neq 0$, \cite[Proposition 20.5]{JL} implies that
		$\chi|_H$ is an irreducible character of $H$. Furthermore, $\chi(1)=\chi|_H(1)$ and $\chi|_H(S\setminus\{x\})=\chi(S\setminus\{x\})$ which imply $\lambda=\theta\pm 1$, for some eigenvalue $\theta$ of $\Gamma_1$. Now assume that $\mu$ be an eigenvalue of $\Gamma_1$. Then $\mu=\frac{\eta(S\setminus\{x\})}{\eta(1)}$ for some $\eta\in\mathrm{Irr}(H)$. Let $\rho_0$ be the principal character of $\langle x\rangle$. Then $\zeta:=\eta\times\rho_0$ is an irreducible character of $G$ and $\zeta(S)=\eta(S\setminus\{x\})+\rho_0(x)=\eta(S\setminus\{x\})+1$. Hence $\mu=\xi-\frac{1}{\zeta(1)}$, where $\xi=\frac{\zeta(S)}{\zeta(1)}$ is an eigenvalue of $\mathrm{Cay}(G,S)$. Since $\frac{1}{\zeta(1)}\in\mathbb{Q}$, we conclude that the splitting filed of $\mathrm{Cay}(G,S)$ and $\Gamma_1$ are the same.
		This proves $(2)$.
	\end{proof}
	
	\begin{lem}\label{productmain}
		Let $\Gamma_1$ be an undiredted graph with the splitting field $\mathbb F$ and $\Gamma_2$ be an integral undirected graph. Then $\Gamma_1\square\Gamma_2$, $\Gamma_1\otimes\Gamma_2$, and $\Gamma_1\boxtimes\Gamma_2$  have the same splitting field $\mathbb F$. Furthermore, if $\Gamma_2$ is a connected regular graph then $\Gamma_1[\Gamma_2]$ has also the same splitting field $\mathbb F$.
	\end{lem}
	\begin{proof}
		Let the eigenvalues of $\Gamma_1$ and $\Gamma_2$ be $\lambda_1\geq \lambda_2\geq\cdots\geq\lambda_{n_1}$ and $\mu_1\geq \mu_2\geq\cdots\geq\mu_{n_2}$, respectively, where $n_i$ is the number of vertices of $\Gamma_i$. By \cite[Table 4]{BKPS}, the eigenvalues of $\Gamma_1\square\Gamma_2$, $\Gamma_1\otimes\Gamma_2$ and $\Gamma_1\boxtimes\Gamma_2$ are $\lambda_i+\mu_j$, $\lambda_i\mu_j$, and $\lambda_i+\mu_j+\lambda_i\mu_j$, respectively, where $1\leq i\leq n_1$ and $1\leq j\leq n_2$. 
		
		Let $\Gamma_2$ be a connected $k$-regular graph. Then eigenvalues of $\Gamma_1[\Gamma_2]$ are $k+n_2\lambda_i$, $1\leq i\leq n_1$ and $\mu_j$ with multiplicity $n_1$, where $2\leq j\leq n_2$. Since eigenvalues of $\Gamma_2$ are integers, the result follows immediately.
	\end{proof}
	
	\begin{core}\label{productCay}
		Let $\Gamma_1=\mathrm{Cay}(G_1,S_1)$ and $\Gamma_2=\mathrm{Cay}(G_2,S_2)$ be two undirected connected Cayley graphs such that $\Gamma_1$ is $k$-integral and $\Gamma_2$ is integral. Then $\mathrm{Cay}(G_1\times G_2,S)$ is an undirected connected $k$-integral Cayley graph over $G_1\times G_2$, where $S$ is one of the following sets
		
		\begin{itemize}
			\item[$(1)$] $S=\{(s_1,1),(1,s_2)\mid s_1\in S_1,s_2\in S_2\}$,
			\item[$(2)$] $S=\{(s_1,s_2)\mid s_1\in S_1,s_2\in S_2\}$,
			\item[$(3)$] $S=\{(s_1,1),(1,s_2),(s_1,s_2)\mid s_1\in S_1,s_2\in S_2\}$,
			\item[$(4)$] $S=\{(s_1,g_2), (1,s_2)\mid s_1\in S_1,s_2\in S_2,g_2\in G_2\}$, where at most one of the graphs $\Gamma_1$ or $\Gamma_2$ is bipartite.
		\end{itemize}
		
		In particular, if $|S_i|=d_i$, $i=1,2$, then in the above cases, $|S|$ is $d_1+d_2$, $d_1d_2$, $d_1+d_2+d_1d_2$ and $d_1|G_2|+d_2$, respectively.
	\end{core}
	\begin{proof}
		It is a direct consequence of Lemma \ref{productmain}, \cite[Corollaries 5.3, 5.6, 5.10, 5.14]{HIK} and \cite[Theorems 3.1,  4.1,  5.3,  6.1]{AL}.
	\end{proof}
	
	\begin{core}\label{odd}
		Let $G$ be a finite abelian group of even order. There exists an undirected connected $(2d+1)$-regular Cayley graph over $G$ with algebraic degree $k$ if and only if 
		\begin{itemize}
			\item [$(1)$] there exists an undirected connected $2d$-regular $k$-integral Cayley graph over $G$, or
			\item [$(2)$] there exists a $H\leq G$ such that $G=H\times\mathbb Z_2$ and $H$ admits an undirected connected $2d$-regular $k$-integral Cayley graph.
		\end{itemize} 
	\end{core}
	\begin{proof}
		Note that the number of involutions of any finite abelian group with even order is odd.
		If $\mathrm{Cay}(G,S)$ is connected with algebraic degree $k$ and $|S|=2d$, for some inverse-closed subset $S$ of $G$, then there exists an involution $x\in G\setminus S$. Hence $\mathrm{Cay}(G,S\cup\{x\})$ is connected $(2d+1)$-regular and with algebraic degree $k$ by Lemma \ref{+1}.		
		If $G=H\times\mathbb Z_2$ and $H$ admits a connected $2d$-regular Cayley graph with algebraic degree $k$, since $K_2$ as a Cayley graph over $\mathbb Z_2$ is integral, then Corollary \ref{productCay} implies that $G$ admits a connected $(2d+1)$-regular Cayley graph with algebraic degree $k$.
		
		The converse is clear by Lemma \ref{-1}, because any inverse-closed subset of $G$ with odd number of elements has an involution.
	\end{proof}
	
	
	\begin{thm}\label{5}
		Let $G$ be a finite abelian group. Then $G\in\mathcal{G}_5$ if and only if $G\cong H$ or $H\times\mathbb Z_2$, where $H$ is one of the groups given in Theorem \ref{4} and in the first case $H\ncong\mathbb Z_{15},\mathbb Z_5^2$.
		Furthermore, a connected $5$-regular abelian undirected Cayley graph $\Gamma$ is $2$-integral if and only if it is isomorphic to one of the  following $108$ Cayley graphs $\mathrm{Cay}(G,S)$, where
		\begin{itemize}
			\item[$\romannumeral1$.] $G=\langle x\rangle \cong \mathbb{Z}_n$ where $n=16,20,24,30$, and $S=\{x,x^{-1},x^{\frac{n}{2}},x^k,x^{-k}\}$ where $1\leq k\leq n$ such that $(k,n)=1$;
			\item[$\romannumeral2$.] $G=\langle x\rangle \cong \mathbb{Z}_8$ and $S=\{x,x^2,x^4,x^6,x^7\}$;		
			\item[$\romannumeral3$.] $G=\langle x\rangle \cong \mathbb{Z}_{10}$ and $S=\{x,x^2,x^5,x^8,x^9\}$;
			\item[$\romannumeral4$.] $G=\langle x\rangle \cong \mathbb{Z}_{12}$ and $S=\{x,x^2,x^6,x^{10},x^{11}\}$ or $\{x,x^3,x^6,x^9,x^{11}\}$ or $\{x,x^4,x^6,x^8,x^{11}\}$;
			\item[$\romannumeral5$.] $G=\langle x\rangle\times \langle w\rangle\cong \mathbb{Z}_n\times \mathbb{Z}_2$ where $n=8,10,12$ and $S=\{x,x^{-1},x^{\frac{n}{2}},w,x^{\frac{n}{2}}w\}$;
			\item[$\romannumeral6$.] $G=\langle x\rangle\times \langle z\rangle\times \langle w\rangle\cong\mathbb{Z}_n\times \mathbb{Z}_2^2$ where $n=5,8,10,12$ and $S=\{x,x^{-1},z,w,s\}$ where $s$ is any involution of $G$ except for $z$ and $w$;
			\item[$\romannumeral7$.] $G=\langle x\rangle\times \langle w\rangle\cong \mathbb{Z}_n\times \mathbb{Z}_m$ where $(n,m)$ is one of the pairs $(3,12)$, $(4,8)$, $(4,10),(4,12)$, $(6,8)$, $(6,10)$, $(6,12)$, $(5,10)$, $(8,8)$, $(10,10)$, $(12,12)$, and $S=\{x,x^{-1},y,y^{-1},s\}$ where $s$ is any involution of $G$.
			\item[$\romannumeral8$.] $G=\langle x\rangle \times \langle y\rangle\cong \mathbb{Z}_n\times \mathbb{Z}_2$ where $n=16,20,24,30$, and $S=\{x,x^{-1},x^k,x^{-k},y\}$ where $1\leq k\leq n$ such that $(k,n)=1$;
			\item[$\romannumeral9$.] $G=\langle x\rangle \times \langle y\rangle\cong \mathbb{Z}_8\times \mathbb{Z}_2$ and $S=\{x,x^2,x^6,x^7,y\}$;		
			\item[$\romannumeral10$.] $G=\langle x\rangle\times \langle y\rangle \cong \mathbb{Z}_{10}\times \mathbb{Z}_2$ and $S=\{x,x^2,x^8,x^9,y\}$;
			\item[$\romannumeral11$.] $G=\langle x\rangle\times \langle y\rangle \cong \mathbb{Z}_{12}\times \mathbb{Z}_2$ and $S=\{x,x^2,x^{10},x^{11},y\}$ or $\{x,x^3,x^9,x^{11},y\}$ or $\{x,x^4,x^8,x^{11},y\}$;
			\item[$\romannumeral12$.] $G=\langle x\rangle\times \langle w\rangle\times \langle y\rangle\cong \mathbb{Z}_n\times \mathbb{Z}_2^2$ where $n=8,10,12$ and $S=\{x,x^{-1},x^{\frac{n}{2}},w,y\}$;
			\item[$\romannumeral13$.] $G=\langle x\rangle\times \langle z\rangle\times \langle w\rangle\times\langle y\rangle\cong \mathbb{Z}_n\times \mathbb{Z}_2^3$ where $n=5,8,10,12$ and $S=\{x,x^{-1},z,w,y\}$;
			\item[$\romannumeral14$.] $G=\langle x\rangle\times \langle w\rangle\times\langle y\rangle\cong \mathbb{Z}_n\times \mathbb{Z}_m\times \mathbb{Z}_2$ where $(n,m)$ is one of the pairs $(3,12)$, $(4,8)$, $(4,10)$, $(4,12)$, $(6,8)$, $(6,10)$, $(6,12)$, $(5,10)$, $(8,8)$, $(10,10)$, $(12,12)$, and $S=\{x,x^{-1},y,y^{-1},y\}$.
		\end{itemize}
	\end{thm}
	\begin{proof}
		Note that in any undirected graph, the number of vertices of odd degree is even.
		Thus $G$ must be a group of even order.
		And so this result is a direct consequence of Theorem \ref{4} and Corollary \ref{odd}.
	\end{proof}
	
	Similar to Theorem \ref{4}, one can find all finite abelian groups $G\in\mathcal{G}_6$ and then it is possible to characterize all finite abelian groups $G\in\mathcal{G}_7$. Recursively, one can find all finite abelian groups in $\mathcal{G}_k$ for any given integer $k$. For instance, if the degree is odd, then Corollary \ref{odd} can be used.
	\section{On groups all of whose undirected Cayley graphs of bounded valency are $2$-integral}		
	In this section, we are going to classify the finite groups $G$ that all undirected Cayley graphs $\mathrm{Cay}(G, S)$ are $2$-integral when $2\leq |S|\leq k$ for each integer $k\geq 2$.	
	This leads us to  give the following definition.
	
	\begin{defi}
		For an integer $k\geq 2$, set
		$$\mathcal{B}_k:=\{G\mid \mathrm{Cay}(G,S)\ \text{is $2$-integral whenever $S \subset G$ such that $1_G\notin S=S^{-1}$ and $2\leq |S|\leq k$}\}.$$
		Clearly, $\mathcal{B}_{k+1}\subseteq \mathcal{B}_k$, and if $|G|$ is odd, then $G\in \mathcal{B}_{2k-1}$ if and only if $G\in \mathcal{B}_{2k-2}$.
	\end{defi}
	
	To determine the sets $\mathcal{B}_k$, we start with some basic and useful results about $\mathcal{B}_k$.	
	
	\begin{lem}\label{basic}
		Let $k\geq 2$ be an integer.
		Then the following holds for every $G\in\mathcal{B}_k$.
		\begin{itemize}
			\item[$1)$] Every proper subgroup $H< G$ with $H\ncong \mathbb{Z}_2$ is also in $\mathcal{B}_k$.
			\item[$2)$] For every $g\in G$, the order of $g$ is in $\{1,2,5,10\}$.
				\item[$3)$] The Sylow $5$-subgroup of $G$ is a $5$-group of exponent $5$.
			\item[$4)$] If $|G|$ is even, then the Sylow $2$-subgroup of $G$ is isomorphic to $\mathbb{Z}_2$.
		
		\end{itemize}
		
		Moreover, $|G|=2^i5^j$, where $i=0,1$ and $j\geq 1$, and the Sylow $5$-subgroup of $G$ is normal. 
	\end{lem}
	\begin{proof}
		Let $G\in\mathcal{B}_k$.
		Suppose $H$ is a proper subgroup of $G$ and $H\ncong \mathbb{Z}_2$.
		Since for a subset $S\subset H\leq G$, the Cayley graph $\mathrm{Cay}(G,S)$ consists of disjoint union  $|G:H|$ copies of $\mathrm{Cay}(H,S)$, we have $1)$. 
		
		Since $G\cong \mathbb{Z}_2$ is Cayley integral, we may assume that $|G|\geq 3$.
		If all non-identity elements of $G$ are involutions, then we have $2)$.
		Next suppose that $G$ has at least one element with order great than $2$, say $g$.
		Then $\mathrm{Cay}(G,\{g,g^{-1}\})$ is $2$-integral and so $\mathrm{Cay}(\langle g\rangle,\{g,g^{-1}\})$ is $2$-integral by 1), which implies $o(g)$,the order of $g$, is in $\{5,8,10,12\}$ by Theorem \ref{2}. If $o(g)=8$ or $12$, then $h:=g^{2}$ or $h:=g^3$ has order $4$, respectively. Now $\mathrm{Cay}(G,\{h,h^{-1}\})$ must be $2$-integral, which means $o(h)\in\{5,8,10,12\}$, a contradiction. This proves $2)$.
		
		
		By $2)$ we may assume that $|G|=2^i5^j$,  $i,j\geq 0$. If $j=0$ then, by $2)$, $G$ is $2$-elementary abelian and so it is Cayley integral, a contradiction. Hence we may assume that $j\geq 1$ and again by $2)$, a Sylow $5$-subgroup of $G$ is a $5$-group of exponent $5$, which proves $3)$. Let $i\neq 0$ and  $P_2$ be a Sylow $2$-subgroup  of $G$. Clearly, $P_2\cong \mathbb{Z}_2^i$.
		If $i\geq 2$, then $P_2\cong \mathbb{Z}_2^i\in \mathcal{B}_k$ by $1)$.
		However $\mathbb{Z}_2^i$ is Cayley integral, a contradiction.
		Thus $i=1$, which proves $4)$.
		
		Furthermore, $G=P_5$ if $i=0$ and $|G:P_5|=2$ if $i=1$, which implies that $P_5$ is normal in $G$.
		This completes the proof.
	\end{proof}
	
	\begin{thm}\label{thm:B4 empty}
		$\mathcal{B}_k$ is an empty set for any $k\geq 4$. 
	\end{thm}
	\begin{proof}
		Suppose $G\in \mathcal{B}_k$ for a $k\geq 4$.
		Then $G\in \mathcal{B}_4$.
		By Lemma \ref{basic}, $G$ has at least one subgroup $H=\langle h\rangle \cong \mathbb{Z}_5$. 
		Note that $\mathrm{Cay}(H,H\setminus \{1\})$ is integral, contradicting to Lemma \ref{basic}--1).
		Thus $G\notin \mathcal{B}_k$ for any $k\geq 4$.
		This completes the proof.
	\end{proof}
	
	\begin{pro}\label{pro:D10}
		$D_{2n}\in \mathcal{B}_k$ for $k=2,3$ if and only if $n=5$.
	\end{pro}
	\begin{proof}
		If $D_{2n}\in \mathcal{B}_k$, then by Lemma \ref{basic} we have $n=5$.
		For the converse direction, it can be directly calculate that $D_{10}\in \mathcal{B}_k$ for $k=2,3$.
	\end{proof}
	 By Theorem \ref{thm:B4 empty}, it is enough to determine the sets $\mathcal{B}_2$ and $\mathcal{B}_3$. In what follows, we will do this.
	\subsection{$\mathcal{B}_2$}
	In this part, we will classify the groups in $\mathcal{B}_2$.
	
	\begin{thm}\label{thm:B2}
		$G\in \mathcal{B}_2$ if and only if $G$ is a $5$-group of exponent $5$ or the Sylow $5$-subgroup of $G$ is a group of exponent $5$ and has index $2$.  
	\end{thm}
	\begin{proof}
		One direction is clear by Lemma \ref{basic}.
		
		Next suppose $P_5$ is a $5$-group of exponent $5$.
		If $G=P_5$, then for any inverse-closed subset $S$ of $G$, $S=\{g,g^{-1}\}$, where the order of $g$ is $5$, and so $\mathrm{Cay}(\langle g\rangle,\{g,g^{-1}\})$ is $2$-integral by Lemma \ref{2}.
		Hence $\mathrm{Cay}(G,S)$ is $2$-integral.
		
		Suppose $P_5<G$ and $|G:P_5|=2$. 
		Then there exist $a\in G$ such that $G=P_5\cup aP_5=P_5\cup P_5a$. Furthermore, by Lemma \ref{basic}, we may assume that $a^2=1$.
		Next we consider the form of inverse-closed subset $S$ of $G$ with $|S|=2$.
		
		{\bf Case I:} $S=\{x,x^{-1}\}$, where $x\in G$ is not an involution.
		
		If $x\in P_5$, then the order of $x$ is $5$ and so $\mathrm{Cay}(\langle x\rangle,\{x,x^{-1}\})$ is $2$-integral by Lemma \ref{2}.
		Hence $\mathrm{Cay}(G,S)$ is $2$-integral.
		
		Suppose that $x\notin P_5$, that is to say $x=ap=p^{'}a$ for some $p,p'\in P_5$. Then $x^2=p'aap=p'p\in P_5$, which means that the order of $x$ is $5$ or $10$. In both cases $\mathrm{Cay}(\langle x\rangle,\{x,x^{-1}\})$ is $2$-integral and so $\mathrm{Cay}(G,S)$ is $2$-integral.

		
		{\bf Case II:} $S=\{x,y\}$, where $x,y\in G$ are both involution.
		
		In this case, $x=ap=p^{-1}a$ and $y=ap'^{-1}=p'a$ for some $p,p'\in P_5$. Note that $xy\neq yx$, otherwise $\langle x,y\rangle\cong \mathbb{Z}_2^2\leq G$ which is impossible by Lemma \ref{basic}--3). Furthermore, $xy\in P_5$. Hence either $xy=1$ or $xy$ has order $5$. The first case is impossible, otherwise $x=y^{-1}$ which is a contradiction. Hence $xy$ has order $5$ and  $\langle x,y\rangle\cong D_{10}$.
		By Proposition \ref{pro:D10}, $D_{10}\in\mathcal{B}_2$.
		Hence $\mathrm{Cay}(G,S)$ is $2$-integral, which completes the proof.
	\end{proof}
	
	\subsection{$\mathcal{B}_3$}
	In this part, we will classify the groups in $\mathcal{B}_3$. First we consider the nilpotent groups.
	\begin{pro}\label{pro:nipoltentB3}
		If $G$ is a nilpotent group, then $G\in \mathcal{B}_3$ if and only if $G\cong P_5$ or $P_5\times \mathbb{Z}_2$, where $P_5$ is a $5$-group of exponent $5$.
	\end{pro}
	\begin{proof}
		Suppose $G$ is a nilpotent group.
		If $G\in \mathcal{B}_3$, then by Lemma \ref{basic}, $G\cong P_5$ or $P_5\times \mathbb{Z}_2$, where $P_5$ is a $5$-group of exponent $5$.
		
		Next suppose $G\cong P_5$ or $G\cong \mathbb{Z}_2\times P_5$, where $P_5$ is a $5$-group of exponent $5$.
		First assume that $G\cong P_5$.
		Then by Theorem \ref{thm:B2}, $P_5\in \mathcal{B}_2$.
		Since $|P_5|$ is odd, we have $G\in \mathcal{B}_3$.
		Assume that $G\cong \mathbb{Z}_2\times P_5$, and $S$ is an inverse-closed subset of $G$ such that $1_G\notin S$ where $|S|=2$ or $3$.
		Note that the order of any non-identity element of $G$ is $2,5$ or $10$.
		If $|S|=2$, then $S=\{g,g^{-1}\}$, where the order of $g$ is $5$ or $10$.
		Furthermore, $\langle S\rangle=\langle g\rangle\cong \mathbb{Z}_5$ or $\mathbb{Z}_{10}$.
		By Lemma \ref{2}, $\mathrm{Cay}(\langle S\rangle,S)$ is $2$-integral.
		Hence $\mathrm{Cay}(G,S)$ is $2$-integral.
		If $|S|=3$, then $S=\{a,g,g^{-1}\}$, where $a^2=1_G$ and the order of $g$ is $5$ or $10$.	
		Furthermore, $\langle S\rangle=\langle a\rangle\times \langle g\rangle \cong \mathbb{Z}_5\times \mathbb{Z}_2=\mathbb{Z}_{10}$.
		By Theorem \ref{3}--$i)$, $\mathrm{Cay}(\langle S\rangle,S)$ is $2$-integral.
		Hence $\mathrm{Cay}(G,S)$ is $2$-integral.
		This completes the proof.
	\end{proof}
	
	In what follows, we consider the non-nilpotent groups in $\mathcal{B}_3$.
	\begin{pro}\label{pro:Z5n}
		Suppose $G=\mathbb{Z}_5^2\rtimes \mathbb{Z}_2$ is not nilpotent.
		Then $G\notin \mathcal{B}_3$.
		Moreover, $G=\mathbb{Z}_5^n\rtimes \mathbb{Z}_2\notin \mathcal{B}_3$ for any $n\geq 2$.
	\end{pro}
	\begin{proof}
		Suppose $G=\mathbb{Z}_5^2\rtimes \mathbb{Z}_2\cong (\langle a\rangle \times \langle b\rangle) \rtimes \langle x\rangle$, where $a^5=b^5=x^2=1_G$.
		Write $P=\langle a\rangle \times \langle b\rangle$.
		Actually, $G$ is the generalized dihedral group over $P$, that is, $p^x=p^{-1}$ for all $p\in P$.
		
		Let $S=\{bx,b^2x,ax\}$.
		Clearly $S$ is inverse-closed.
		Next we will show that $\mathrm{Cay}(G,S)$ is not $2$-integral.
		Note that $\mathrm{Cay}(G,S)$ is isomorphic to $\mathrm{BiCay}(P,\emptyset,\emptyset,T)$, where $T=\{b,b^2,a\}$, according \cite[Lemma 8]{AT}.
		By \cite[Theorem 3.2]{bicay}, the eigenvalues of $\mathrm{BiCay}(P,\emptyset,\emptyset,T)$ are $\pm |\chi_p(b+b^2+a)|$, where $p\in P$ and $\chi_p$'s are irreducible characters of $P$, see also \cite[Section 3.]{A-PST}.
		However, one can calculate that $\pm |\chi_{b^4}(b+b^2+a)|=\sqrt{4+\sqrt{5}}$.
		Thus $\mathrm{BiCay}(P,\emptyset,\emptyset,T)$ is not $2$-integral.
		Hence $G\notin \mathcal{B}_3$.
		
		On the other hand, $\mathbb{Z}_5^2\rtimes \mathbb{Z}_2\leq \mathbb{Z}_5^n\rtimes \mathbb{Z}_2$ for $n\geq 2$.
		By Lemma \ref{basic}--1), if $\mathbb{Z}_5^n\rtimes \mathbb{Z}_2\in\mathcal{B}_3$, then $\mathbb{Z}_5^2\rtimes \mathbb{Z}_2\in\mathcal{B}_3$, a contradiction.
		Thus $\mathbb{Z}_5^n\rtimes \mathbb{Z}_2\notin \mathcal{B}_3$ for any $n\geq 2$.
	\end{proof}
	
	The structure of minimal non-abelian group of exponent $5$ is given in the following lemma. To have a full characterization of non-nilpotent groups in $\mathcal{B}_3$, we need this lemma.
	\begin{lem}\cite{Redei}\label{lem:G5}
		The minimal non-abelian group of exponent $5$ is $G_5=\langle a,b,c\mid a^5=b^5=c^5=1,~[a,b]=c,~[a,c]=[b,c]=1\rangle$.
	\end{lem}
	 Keeping the notation of Lemma \ref{lem:G5}, we get the following result.
	\begin{lem}\label{lem:G5Z2}
		Suppose	$G=G_5\rtimes \mathbb{Z}_2$ is not nilpotent.
		Then $G$ is isomorphic to one of the following groups:
		\begin{itemize}
			\item[$1)$] $G_{250,1}\cong ((\langle a\rangle \times \langle b\rangle)\rtimes \langle c\rangle)\rtimes \langle d\rangle$, where $a^5=b^5=c^5=d^2=1_G$, $a^d=a^4$, $b^d=b^4$, $c^d=c$ and $b^c=ab$;
			\item[$2)$] $G_{250,2}\cong ((\langle a\rangle \times \langle b\rangle)\rtimes \langle c\rangle)\rtimes \langle d\rangle$, where $a^5=b^5=c^5=d^2=1_G$, $a^d=a$, $b^d=b^4$, $c^d=c^4$ and $b^c=ab$.
		\end{itemize}
		Moreover, $G_{250,1}$ and $G_{250,2}$ are not contained in $\mathcal{B}_3$.
	\end{lem}
	\begin{proof}
		By GAP \cite{GAP}, $G$ is isomorphic to $G_{250,1}$ or $G_{250,2}$. Furthermore, $G_{250,1}$ has a subgroup isomorphic to $\mathbb{Z}_5^2\rtimes \mathbb{Z}_2$ and $G_{250,2}$ has a subgroup isomorphic to $\mathbb{Z}_5\times D_{10}$.
		By Lemma \ref{basic}--1), $\mathbb{Z}_5^2\rtimes \mathbb{Z}_2\notin \mathcal{B}_3$.
		Hence $G_{250,1}\notin \mathcal{B}_3$.
		Suppose $\mathbb{Z}_5\times D_{10}\cong \langle x\rangle\times(\langle y\rangle\rtimes \langle z\rangle)\leq G_{250,2}$, where $x^5=y^5=z^2=1$.
		Let $S=\{y^4z,xy^3z,x^4y^3z\}$.
		Clearly, $S=S^{-1}$ and $\langle S\rangle\cong \mathbb{Z}_5\times D_{10}$.
		By a calculation, $\sqrt{4+\sqrt{5}}$ is an eigenvalue of $\mathrm{Cay}(\mathbb{Z}_5\times D_{10},S)$.
		Thus $\mathbb{Z}_5\times D_{10}\notin \mathcal{B}_3$.
		This implies that $G_{250,2}\notin \mathcal{B}_3$.
	\end{proof}
	 
	 Now we are ready to characterize non-nilpotent groups in $\mathcal{B}_3$.
	\begin{pro}\label{pro:not nilpotent}
		If $G$ is not nilpotent, then $G\in \mathcal{B}_3$ if and only if $G\cong D_{10}$.
	\end{pro}
	\begin{proof}
		Suppose $G$ is not nilpotent and $G\in \mathcal{B}_3$.
		By Lemma \ref{basic}, $G\cong P_5\rtimes \mathbb{Z}_2$, where $P_5$ is a $5$-group of exponent $5$.
		If $P_5$ is abelian, then $P_5\cong \mathbb{Z}_5^n$ for some $n\geq 1$.
		By Proposition \ref{pro:D10} and \ref{pro:Z5n}, $P_5=\mathbb{Z}_5$ and so $G\cong D_{10}$.
		If $P_5$ is not abelian, then $P_5$ has a minimal non-abelian subgroup isomorphic to $G_5$ by Lemma \ref{lem:G5}.
		Moreover, $G$ has a subgroup isomorphic to $G_5\rtimes \mathbb{Z}_2$.
		By Lemma \ref{lem:G5Z2}, $G_5\rtimes \mathbb{Z}_2\notin \mathcal{B}_3$, which contradicts Lemma \ref{basic}--1).
		Hence $P_5$ is abelian, and $G\cong D_{10}$. The converse direction is clear by Proposition \ref{pro:D10}. This completes the proof.
	\end{proof}
	
	By Propositions \ref{pro:nipoltentB3} and \ref{pro:not nilpotent}, we have the following result.
	\begin{thm}\label{thm:B3}
		$G\in \mathcal{B}_3$ if and only if $G\cong P_5$ or $P_5\times \mathbb{Z}_2$ or $D_{10}$, where $P_5$ is a $5$-group of exponent $5$.
	\end{thm}
	
	\textbf{Acknowledgments.}
	Research of this paper was carried out while the second author was visiting Beijing Jiaotong University. He expresses his sincere thanks to the 111 Project of China (grant number B16002) for financial support and to the School of Mathematics and Statistics of Beijing Jiaotong University for their kind hospitality. Also he was partially supported by the Shahid Rajaee Teacher Training University. Tao Feng is supported by NSFC under Grants 11871095 and 12271023. Shixin Wang is supported by research program P1-0285 and research project J1-50000 by the Slovenian Research and Innovation Agency.

\end{document}